\documentclass{amsart}
\usepackage{amsfonts}
\usepackage{amsmath}
\usepackage{amssymb}
\usepackage{amsthm}
\usepackage{graphicx}
\usepackage{capt-of}
\usepackage{tikz}
\usetikzlibrary{patterns}
\usetikzlibrary{decorations.markings}
\usepackage{multirow}
\usepackage{array} 
\newcolumntype{C}{>{$}c<{$}}

\newtheorem{theorem}{Theorem}

\newtheorem{lemma}{Lemma}

\newtheorem{exmp}{Example}
\newtheorem{cor}{Corollary}

\begin{document}

\title{On two types of $Z$-monodromy in triangulations of surfaces}
\author{Mark Pankov, Adam Tyc}
\subjclass[2000]{}

\address{Mark Pankov: Faculty of Mathematics and Computer Science, 
University of Warmia and Mazury, S{\l}oneczna 54, 10-710 Olsztyn, Poland}
\email{pankov@matman.uwm.edu.pl}

\address{Adam Tyc: Institute of Mathematics, Polish Academy of Science, \'Sniadeckich 8, 00-656 Warszawa, Poland}
\email{atyc@impan.pl}

\maketitle

\begin{abstract}
Let $\Gamma$ be a triangulation of a connected closed $2$-dimensional (not necessarily orientable) surface.
Using zigzags (closed left-right paths), 
for every face of $\Gamma$ we define the $z$-monodromy which acts on the oriented edges of this face.
There are precisely $7$ types of $z$-monodromies.
We consider  the following two cases: (M1) the $z$-monodromy is identity,
(M2) the $z$-monodromy is the consecutive passing of the oriented edges.
Our main result is the following: 
the subgraphs of the dual graph $\Gamma^{*}$ formed by edges whose $z$-monodromies are of types (M1) and (M2), respectively,
both are forests.
We apply this statement to the connected sum of $z$-knotted triangulations. 
\end{abstract}

\section{Introduction}
A {\it zigzag}  of a graph embedded in a surface is a closed path,
where any two consecutive edges, but not three, belong to a face \cite{DDS-book,Lins1}.
Zigzags are known also as {\it Petrie paths} \cite{Coxeter} or {\it closed left-right paths} \cite{GR-book,Shank}.
Similar objects in simplicial complexes and abstract polytopes are investigated in \cite{DP,W}.
An embedded graph is called $z$-{\it knotted} if it contains a single zigzag. 
Such graphs are closely connected to the {\it Gauss code problem} and have interesting homological properties
(see \cite[Section 17.7]{GR-book} for the planar case and \cite{CrRos,Lins2} for the case when a graph is embedded in an arbitrary surface).
It was established in \cite{PT2} that 
every triangulation of a connected closed 2-dimensional (not necessarily orientable) surface admits a z-knotted shredding.
The concept of $z$-{\it monodromy} is a crucial tool used in the construction of  such shreddings.

For every face $F$ of an embedded graph we define the $z$-monodromy $M_{F}$ which acts on the oriented edges of this face.
If $e$ is such an edge, then there is a unique zigzag coming out from $F$ through $e$ and 
we define $M_{F}(e)$ as the first oriented edge of $F$ which occurs in this zigzag after $e$.

By \cite{PT2}, there are precisely seven types of $z$-monodromies for faces in triangulations (we describe them in Section 3).
A triangulation is $z$-knotted if and only if the $z$-monodromy of every face has one of the four $z$-knotted types.
In this paper, we consider the following two types of $z$-knotted monodromies: (M1) the $z$-monodromy is identity,
(M2) the $z$-monodromy is the consecutive passing of the oriented edges.
We show that for any triangulation the subgraphs of the dual graph formed by faces whose $z$-monodromies are of type (M1) and (M2), respectively,
both are forests. 
Examples of these subgraphs will be considered in Section 7.

In particular, this result implies that every triangulation contains a face whose $z$-monodromy is not identity. 
Combining this fact with \cite[Theorem 4]{PT2}, we get the following:
for any two $z$-knotted triangulations 
there are a pair of faces and a bijection between the vertex sets of these faces 
such that the corresponding connected sum of triangulations is $z$-knotted.

\section{Zigzags in triangulations}
Let $\Gamma$ be a {\it triangulation} of  a connected closed $2$-dimensional (not necessarily orientable) surface $M$,
i.e. a closed $2$-cell embedding of a finite connected graph in $M$ such that the following conditions hold: 
1) every face contains precisely three edges,
2) every edge is contained in precisely  two distinct faces,
3) the intersection of two distinct faces is an edge or a vertex or empty.
See \cite{MT-book} for more information concerning triangulations of surfaces.

A {\it zigzag}  in $\Gamma$ is a sequence of edges $\{e_{i}\}_{i\in {\mathbb N}}$ satisfying the following conditions for every $i\in {\mathbb N}$:\begin{enumerate}
\item[$\bullet$] the edges $e_{i}$ and $e_{i+1}$ are adjacent, i.e. they have a common vertex and  there is a face containing them;
\item[$\bullet$] the faces containing $e_{i},e_{i+1}$ and $e_{i+1},e_{i+2}$ are distinct
and the edges $e_{i}$ and $e_{i+2}$ are non-intersecting.
\end{enumerate} 
Since the number of edges is finite, this sequence is cyclic, i.e.
there is a natural number $n>0$ such that $e_{i+n}=e_{i}$ for every $i\in {\mathbb N}$. 
In what follows, every zigzag will be written as a cyclic sequence $e_{1},\dots,e_{n}$,
where $n$ is the smallest number satisfying the above condition.

Every zigzag is completely determined by any pair of consecutive edges
and for any pair of adjacent edges $e$ and $e'$ there is a unique zigzag containing the sequence $e,e'$.
If $Z=\{e_{1},\dots,e_{n}\}$ is a zigzag, then the same holds for the reversed sequence $Z^{-1}=\{e_{n},\dots,e_{1}\}$.
It is not difficult to prove that a zigzag cannot contain a sequence of type $e,e',\dots,e',e$.
For this reason, $Z\ne Z^{-1}$ for any zigzag $Z$, in other words, 
a zigzag cannot be self-reversed.

\begin{exmp}\rm
See Fig.1 for zigzags (drawn by the bold line) in the three Platonic solids which are triangulations of ${\mathbb S}^2$.
\begin{center}
\begin{tikzpicture} [scale=0.5]
\draw[xshift=4.375cm, fill=black] (0:1.75cm) circle (3pt);
\draw[xshift=4.375cm, fill=black] (90:1.75cm) circle (3pt);
\draw[xshift=4.375cm, fill=black] (180:1.75cm) circle (3pt);
\draw[xshift=4.375cm, fill=black] (270:1.75cm) circle (3pt);
    \draw[xshift=4.375cm,thick,line width=2pt] (0:1.75cm) \foreach \x in {90,180,...,359} {
            -- (\x:1.75cm) 
        } -- cycle (90:1.75cm);
    \draw[xshift=4.375cm, dashed] (0:1.75cm) \foreach \x in {90,270} {-- (\x:1.75cm)};
    \draw[xshift=4.375cm] (0:1.75cm) \foreach \x in {0,180} {-- (\x:1.75cm)};

\draw[xshift=8.75cm, fill=black] (0:1.75cm) circle (3pt);
\draw[xshift=8.75cm, fill=black] (60:1.75cm) circle (3pt);
\draw[xshift=8.75cm, fill=black] (120:1.75cm) circle (3pt);
\draw[xshift=8.75cm, fill=black] (180:1.75cm) circle (3pt);
\draw[xshift=8.75cm, fill=black] (240:1.75cm) circle (3pt);
\draw[xshift=8.75cm, fill=black] (300:1.75cm) circle (3pt);
    \draw[xshift=8.75cm,thick,line width=2pt] (0:1.75cm) \foreach \x in {60, 120,...,359} {
            -- (\x:1.75cm) 
        } -- cycle (60:1.75cm);
    \draw[xshift=8.75cm, dashed] (0:1.75cm)--(120:1.75cm)--(240:1.75cm)--cycle;
    \draw[xshift=8.75cm] (60:1.75cm)--(180:1.75cm)--(300:1.75cm)--cycle;

\draw[xshift=13.125cm, fill=black] (0:1.75cm) circle (3pt);
\draw[xshift=13.125cm, fill=black] (36:1.75cm) circle (3pt);
\draw[xshift=13.125cm, fill=black] (72:1.75cm) circle (3pt);
\draw[xshift=13.125cm, fill=black] (108:1.75cm) circle (3pt);
\draw[xshift=13.125cm, fill=black] (144:1.75cm) circle (3pt);
\draw[xshift=13.125cm, fill=black] (180:1.75cm) circle (3pt);
\draw[xshift=13.125cm, fill=black] (216:1.75cm) circle (3pt);
\draw[xshift=13.125cm, fill=black] (252:1.75cm) circle (3pt);
\draw[xshift=13.125cm, fill=black] (288:1.75cm) circle (3pt);
\draw[xshift=13.125cm, fill=black] (324:1.75cm) circle (3pt);
    \draw[xshift=13.125cm,thick,line width=2pt] (0:1.75cm) \foreach \x in {36, 72,...,359} {
            -- (\x:1.75cm) 
        } -- cycle (36:1.75cm);
\draw[xshift=13.125cm, fill=black] (0,0) circle (3pt);
    \draw[xshift=13.125cm] (0:1.75cm)--(0,0) (72:1.75cm)--(0,0) (144:1.75cm)--(0,0) (216:1.75cm)--(0,0) (288:1.75cm)--(0,0);
    \draw[xshift=13.125cm] (0:1.75cm)--(72:1.75cm)--(144:1.75cm)--(216:1.75cm)--(288:1.75cm)--cycle;
    \draw[xshift=13.125cm, dashed] (36:1.75cm)--(0,0) (108:1.75cm)--(0,0) (180:1.75cm)--(0,0) (252:1.75cm)--(0,0) (324:1.75cm)--(0,0);
    \draw[xshift=13.125cm, dashed] (36:1.75cm)--(108:1.75cm)--(180:1.75cm)--(252:1.75cm)--(324:1.75cm)--cycle;
\end{tikzpicture}
\captionof{figure}{ }
\end{center}
The sequence
$$a1, 12, 2b, b3, 31, 1a, a2, 23, 3b, b1, 12, 2a, a3, 31, 1b, b2, 23, 3a$$
is a zigzag in the $3$-gonal bipyramid (Fig.2). This zigzag is unique (up to reverse).
\end{exmp}

Let $F$ be a face in $\Gamma$.
Denote by ${\mathcal Z}(F)$ the set of all zigzags containing edges of $F$.
If a zigzag $Z$ belongs to ${\mathcal Z}(F)$, then the same holds for the reversed zigzag $Z^{-1}$.
Every zigzag from ${\mathcal Z}(F)$ contains at least two edges of $F$. 
We say that the triangulation $\Gamma$ is {\it locally $z$-knotted} in the face $F$ if ${\mathcal Z}(F)$ contains precisely two zigzags,
i.e. ${\mathcal Z}(F)=\{Z,Z^{-1}\}$.

A triangulation is called $z$-{\it knotted} if it contains precisely two zigzags,
in other words, there is a single zigzag (up to reverse).
A triangulation is $z$-knotted if and only if it is locally $z$-knotted in each face
\cite[Theorems 2 and 3]{PT2}. 
\begin{center}
\begin{tikzpicture}[line join = round, line cap = round, scale=0.4]
\pgfmathsetmacro{\factor}{1/sqrt(2)};
\coordinate (A) at (2,0,-2*\factor);
\coordinate (B) at (-2,0,-2*\factor);
\coordinate (C) at (0.75,0.5,2*\factor);
\coordinate (D) at (1.25,-2.5,2*\factor);
\coordinate (E) at (1.25,4.5,2*\factor);

\draw[fill=black] (A) circle (3pt);
\draw[fill=black] (B) circle (3pt);
\draw[fill=black] (C) circle (3pt);
\draw[fill=black] (D) circle (3pt);
\draw[fill=black] (E) circle (3pt);

\draw (A)--(D)--(B)--cycle;
\draw (A) --(D)--(C)--cycle;
\draw (B)--(D)--(C)--cycle;

\draw (A)--(E)--(B)--cycle;
\draw (A) --(E)--(C)--cycle;
\draw (B)--(E)--(C)--cycle;

\node at (2.5,0,-2*\factor) {$3$};
\node at (-2.5,0,-2*\factor) {$1$};
\node at (0.35,0.25,2*\factor) {$2$};

\node at (1.25,-3,2*\factor) {$b$};
\node at (1.25,5,2*\factor) {$a$};
\end{tikzpicture}
\captionof{figure}{ }
\end{center}

\section{$Z$-monodromy}
Let $F$ be a face in $\Gamma$ whose vertices are $a,b,c$.
Consider the set $\Omega(F)$ consisting of  all {\it oriented} edges of $F$, i.e.
$$\Omega(F)=\{ab,bc,ca,ac,cb,ba\},$$
where $xy$ is  the edge from $x\in\{a,b,c\}$ to $y\in\{a,b,c\}$.
If $e=xy$, then we write $-e$ for the edge $yx$.
Denote by $D_{F}$ the permutation
$$(ab,bc,ca)(ac,cb,ba)$$
on the set $\Omega(F)$.
If $x,y,z$ are three mutually distinct vertices of $F$, then $D_{F}(xy)=yz$.
The $z$-{\it monodromy} of the face $F$ is the permutation $M_{F}$ defined on $\Omega(F)$ as follows.
For any $e\in \Omega(F)$ we take $e_{0}\in \Omega(F)$ such that $D_{F}(e_{0})=e$
and consider the zigzag containing the sequence $e_{0},e$. 
We define $M_{F}(e)$ as the first element of  $\Omega(F)$ contained in this zigzag after $e$.

\begin{exmp}{\rm
It is easy to see that for every face $F$ in each of the three Platonic triangulations (Example 1) the $z$-monodromy is $D^{-1}_{F}$.
More generally, the following two conditions are equivalent:
\begin{enumerate}
\item[$\bullet$] every zigzag in a triangulation is edge-simple, i.e. all edges in this zigzag are mutually distinct;
\item[$\bullet$] $M_{F}=D^{-1}_{F}$ for every face $F$.
\end{enumerate}
See \cite[Example 1]{PT2}.
}\end{exmp}

By \cite[Theorem 2]{PT2},  we have the following possibilities for the $z$-monodromy $M_{F}$:
\begin{enumerate}
\item[(M1)] $M_{F}$ is identity,
\item[(M2)] $M_{F}=D_{F}$,
\item[(M3)] $M_{F}=(-e_{1},e_{2},e_{3})(-e_{3},-e_{2},e_{1})$, where $(e_{1},e_{2},e_{3})$ is one of the cycles in the permutation  $D_{F}$,
\item[(M4)] $M_{F}=(e_{1},-e_{2})(e_{2},-e_{1})$, where $(e_{1},e_{2},e_{3})$ is one of the cycles in $D_{F}$
{\rm(}$e_{3}$ and $-e_{3}$ are fixed points{\rm)},
\item[(M5)] $M_{F}=(D_{F})^{-1}$,
\item[(M6)] $M_{F}=(-e_{1},e_{2},e_{3})(-e_{3},-e_{2},e_{1})$, where $(e_{1},e_{2},e_{3})$ is one of the cycles in the permutation $(D_{F})^{-1}$,
\item[(M7)] $M_{F}=(e_{1},e_{2})(-e_{1},-e_{2})$, where $(e_{1},e_{2},e_{3})$ is one of the cycles in $D_{F}$
{\rm(}$e_{3}$ and $-e_{3}$ are fixed points{\rm)}.
\end{enumerate}
The triangulation $\Gamma$ is locally $z$-knotted in the face $F$ only in the cases (M1)--(M4).
A triangulation is $z$-knotted if and only if the $z$-monodromies of all faces are types (M1)--(M4) \cite[Theorem 3]{PT2}.

Each of the possibilities (M1)--(M7) is realized \cite[Section 5]{PT2}. 
In particular, the following assertions are fulfilled: 
\begin{enumerate}
\item[$\bullet$] The $z$-monodromy of every face in the $(2k+1)$-gonal bipyramid is of type (M3) if $k$ is odd;
in the case when $k$ is even, the $z$-monodromy of every face is of type (M4). 
\item[$\bullet$] The $z$-monodromy of every face in the $2k$-gonal bipyramid is of type (M7) if $k$ is odd;
if $k$ is even, then the $z$-monodromies of all faces are of type (M5).
\end{enumerate}
Examples of triangulations containing faces with the $z$-monodromies of types (M1) and (M2) will be considered in Section 7.

Suppose that $\Gamma$ is locally $z$-knotted in a face $F$.
Then for the $z$-monodromy $M_{F}$ one of the possibilities (M1)--(M4) is realized.
 If $M_{F}$ is of type (M1) or (M2), then every zigzag from ${\mathcal Z}(F)$
 passes through each edge of $F$ twice in the same direction; more precisely, 
 it goes twice through three elements of $\Omega(F)$ which form a cycle in $D_{F}$ (Fig.3a).
\begin{center}
\begin{tikzpicture}[scale=0.6]
\draw[fill=black] (0,2) circle (3pt);
\draw[fill=black] (-1.7320508076,-1) circle (3pt);
\draw[fill=black] (1.7320508076,-1) circle (3pt);

\draw [thick, decoration={markings,
mark=at position 0.65 with {\arrow[scale=2,>=stealth]{>>}}},
postaction={decorate}] (0,2) -- (-1.7320508076,-1);

\draw [thick, decoration={markings,
mark=at position 0.65 with {\arrow[scale=2,>=stealth]{>>}}},
postaction={decorate}] (-1.7320508076,-1) -- (1.7320508076,-1);

\draw [thick, decoration={markings,
mark=at position 0.65 with {\arrow[scale=2,>=stealth]{<<}}},
postaction={decorate}] (0,2) -- (1.7320508076,-1);

\node at (0,-1.6) {(a)};

\draw[fill=black] (5.1961524228,2) circle (3pt);
\draw[fill=black] (3.4641016152,-1) circle (3pt);
\draw[fill=black] (6.9282032304,-1) circle (3pt);

\draw [thick, decoration={markings,
mark=at position 0.65 with {\arrow[scale=2,>=stealth]{><}}},
postaction={decorate}] (5.1961524228,2) -- (3.4641016152,-1);

\draw [thick, decoration={markings,
mark=at position 0.65 with {\arrow[scale=2,>=stealth]{>>}}},
postaction={decorate}] (3.4641016152,-1) -- (6.9282032304,-1);

\draw [thick, decoration={markings,
mark=at position 0.65 with {\arrow[scale=2,>=stealth]{><}}},
postaction={decorate}] (5.1961524228,2) -- (6.9282032304,-1);

\node at (5.1961524228,-1.6) {(b)};
\end{tikzpicture}
\captionof{figure}{ }
\end{center}
If $M_{F}$ is of type (M3) or (M4), then
every zigzag from ${\mathcal Z}(F)$ goes through one edge twice in the same direction and through the remaining two edges twice in different directions (Fig.3b).
If $\Gamma$ is $z$-knotted, then  there is a single zigzag (up to reverse) which passes through every edge of $\Gamma$ twice and 
each face of $\Gamma$ has one of the types described on Fig.3.

\section{Main result}
Consider the dual graph $\Gamma^{*}$ whose vertices are faces of $\Gamma$ and whose edges are pairs of faces intersecting in an edge.
Denote by ${\mathrm G}_{1}$ the subgraph in $\Gamma^{*}$ consisting of all faces whose $z$-monodromies are identity, i.e. of type (M1);
two such faces are adjacent vertices of ${\mathrm G}_{1}$ if they are adjacent vertices of $\Gamma^{*}$.
Similarly, we define the subgraph ${\mathrm G}_{2}$ formed by all faces with $z$-monodromies of type (M2).
Our main result is the following.

\begin{theorem}
The graphs ${\mathrm G}_{1}$ and ${\mathrm G}_2$ are forests.
\end{theorem}

Let $\Gamma_{1}$ and $\Gamma_2$ be triangulations of connected closed $2$-dimensional surfaces $M_1$ and $M_{2}$,
respectively.
Consider a face $F_{1}$ in $\Gamma_{1}$, a face $F_{2}$ in $\Gamma_{2}$
and a homeomorphism $g: \partial F_{1} \to \partial F_{2}$ which sends every vertex of $F_{1}$ to a vertex of $F_{2}$, i.e.
if $v_{i}$, $i\in \{1,2,3\}$ are the vertices of $F_{1}$, then $w_{i}=g(v_{i})$, $i\in \{1,2,3\}$ are the vertices of $F_{2}$.
Such boundary homeomorphisms are said to be  {\it special}.
The {\it connected sum} $\Gamma_{1} \#_{g} \Gamma_{2}$ is the graph whose vertex set is the union of the vertex sets of 
$\Gamma_{1}$ and $\Gamma_{2}$,
where each $v_i$ is identified with $w_i$, and the edge set is the union of the edge sets of $\Gamma_{1}$ and $\Gamma_{2}$, 
where the edge $v_iv_j$ is identified with the edge $w_iw_j$.
This is a triangulation of the connected sum of the surfaces $M_{1}$ and $M_{2}$.
For other special homeomorphism $h: \partial F_{1} \to \partial F_{2}$ the graph $\Gamma_{1} \#_{h} \Gamma_{2}$ 
is not necessarily isomorphic to $\Gamma_{1} \#_{g} \Gamma_{2}$.

In the case when $\Gamma_{1}$ and $\Gamma_2$ are $z$-knotted and the $z$-monodromies of $F_{1}$ and $F_{2}$ are not identity,
there is a special homeomorphism $g: \partial F_{1} \to \partial F_{2}$ such that the connected sum $\Gamma_{1} \#_{g} \Gamma_{2}$ is $z$-knotted
 \cite[Theorem 4]{PT2}.
By Theorem 1, each triangulation contains a face whose $z$-monodromy is not identity. 
So, we get the following.

\begin{cor}
For any $z$-knotted triangulations  $\Gamma_{1}$ and $\Gamma_{2}$ of connected closed $2$-dimensional surfaces 
there are faces $F_{1}$ and $F_{2}$ in $\Gamma_{1}$ and $\Gamma_{2}$ {\rm(}respectively{\rm)} 
and a special homeomorphism $g: \partial F_{1} \to \partial F_{2}$ such that 
the connected sum $\Gamma_{1} \#_{g} \Gamma_{2}$ is a $z$-knotted triangulation.
\end{cor}

\section{The graph ${\mathrm G}_{1}$ is a forest}

The {\it face shadow} of a zigzag $Z=\{e_{1},\dots,e_{n}\}$ is a cyclic sequence of faces $F_{1},\dots, F_{n}$, 
where $F_{i}$ is the face containing the edges $e_{i}$ and $e_{i+1}$.
Any two consecutive faces in this sequence are adjacent.

\begin{lemma}\label{l-1} 
Let $F$ be a face belonging to the face shadow $F_{1},\dots, F_{n}$ of a certain zigzag.
Then there is at most three distinct indices  $i$ such that $F_{i}=F$.
If our triangulation is locally $z$-knotted in $F$, then there are precisely three such $i$.
\end{lemma}

\begin{proof}
For every edge $e\in \Omega(F)$ we denote by $Z(e)$ the zigzag containing the sequence $e,D_{F}(e)$.
Observe that $Z(e')=Z(e)^{-1}$ if $e'=-D_{F}(e)$.
Also, it can be happened that  $Z(e)=Z(e')$ for some distinct $e,e'\in \Omega(F)$.
This means that ${\mathcal Z}(F)$ contains at most three pairs of zigzags $Z,Z^{-1}$.
Therefore, if $F_{1},\dots, F_{n}$ is the shadow of a zigzag from ${\mathcal Z}(F)$ and $F_{i}=F$ for four distinct indices $i$,
then this zigzag is self-reversed which is impossible.
In the case when the triangulation is locally $z$-knotted in $F$, 
for any two $e,e'\in \Omega(F)$ we have $Z(e)=Z(e')$ or $Z(e)=Z(e')^{-1}$ which implies the second statement.
\end{proof}

\begin{lemma}\label{l-2}
Let $F$ and $F'$ be adjacent faces whose $z$-monodromies both are identity.
Then there is a unique {\rm(}up to reverse{\rm)} zigzag whose face shadow contains $F$ and $F'$.
This face shadow is a cyclic sequence of type 
$$F,F',\dots,F',\dots,F',F,\dots,F,\dots$$
(Fig.4)\footnote{The reversed sequence
$F',F,\dots,F,\dots,F,F',\dots,F',\dots$
is the face shadow of the reversed zigzag.}.
\end{lemma}

\begin{center}
\begin{tikzpicture}[scale=0.6]
\draw (0,0) circle (2cm);
\draw[fill=black] (0:2cm) circle (3pt);
\draw[fill=black] (180:2cm) circle (3pt);

\draw[fill=black] (86:2cm) circle (3pt);
\draw[fill=black] (94:2cm) circle (3pt);
\draw[fill=black] (266:2cm) circle (3pt);
\draw[fill=black] (274:2cm) circle (3pt);

\node at (0:2.5cm) {$F'$};
\node at (180:2.5cm) {$F$};

\node at (84:2.5cm) {$F'$};
\node at (96:2.5cm) {$F$};
\node at (264:2.5cm) {$F$};
\node at (276:2.5cm) {$F'$};

\end{tikzpicture}
\captionof{figure}{ }
\end{center}

\begin{proof}
Let $x,y,z$ and $t,y,z$ be the vertices of $F$ and $F'$ (respectively)
and let
$$e_1=yz,\;\; e_2=zx,\;\; e_3=xy,\;\; e'_2=zt,\;\; e'_3=ty$$
(Fig.5).
The intersection of $\Omega(F)$ and $\Omega(F')$ is $\{e_{1},-e_{1}\}$.
\begin{center}
\begin{tikzpicture}[scale=0.8]
\draw[fill=black] (90:1.5) circle (3pt);
\draw[fill=black] (210:1.5) circle (3pt);
\draw[fill=black] (330:1.5) circle (3pt);
\draw[fill=black] (-90:3) circle (3pt);

\draw [black,line width=1.25pt]  (90:1.5) -- (210:1.5) -- (-90:3) -- (330:1.5)  -- cycle;
\draw [black,line width=1.25pt]  (330:1.5) -- (210:1.5);

\node at (0, 0.1) {$F$};
\node at (0, -1.6) {$F'$};

\draw [thick, decoration={markings,
mark=at position 0.57 with {\arrow[scale=2,>=stealth]{>}}},
postaction={decorate}] (90:1.5) -- (210:1.5);

\draw [thick, decoration={markings,
mark=at position 0.57 with {\arrow[scale=2,>=stealth]{>}}},
postaction={decorate}] (-90:3) -- (210:1.5);

\draw [thick, decoration={markings,
mark=at position 0.57 with {\arrow[scale=2,>=stealth]{>}}},
postaction={decorate}] (330:1.5) -- (-90:3);

\draw [thick, decoration={markings,
mark=at position 0.57 with {\arrow[scale=2,>=stealth]{>}}},
postaction={decorate}] (330:1.5) -- (90:1.5);

\draw [thick, decoration={markings,
mark=at position 0.57 with {\arrow[scale=2,>=stealth]{>}}},
postaction={decorate}] (210:1.5) -- (330:1.5);

\node at (-0.3,1.5) {$x$};
\node at (-1.6, -0.75) {$y$};
\node at (1.6, -0.75) {$z$};
\node at (-0.3,-3) {$t$};

\node at (0,-0.45) {$e_1$};
\node at (0.95,0.5) {$e_2$};
\node at (-0.95,0.5) {$e_3$};
\node at (0.95,-2) {$e'_2$};
\node at (-0.95,-2) {$e'_3$};

\end{tikzpicture}
\captionof{figure}{ }
\end{center}
Since the $z$-monodromies $M_{F}$ and $M_{F'}$ both are identity,
the triangulation $\Gamma$ is locally $z$-knotted in $F$ and $F'$. 
The faces $F$ and $F'$ are adjacent and we have
$${\mathcal Z}(F)={\mathcal Z}(F')=\{Z,Z^{-1}\}.$$
Suppose that $Z$ is the zigzag containing the sequence $e_{3},e_{1},e'_{2}$.

We determine the first edge $e$ from $\Omega(F)\cup \Omega(F')$ which occurs in the zigzag $Z$ after this sequence.
If this edge belong to $\Omega(F)$, then it coincides with $M_{F}(e_{1})=e_{1}$ which is impossible
(we can come to $e_{1}$ by a zigzag only through an element of $\Omega(F)$ or $\Omega(F')$ different from $e_1$).
Therefore, $e$ belongs to $\Omega(F')$. This implies that $e=M_{F'}(e'_{2})=e'_{2}$.
The next edge of $Z$ is $D_{F'}(e'_{2})=e'_{3}$ and the zigzag $Z$ is a cyclic sequence of type 
$$e_{3},e_{1},e'_{2},X,e'_{2},e'_{3},\dots,$$
where $X$ is a sequence of edges which does not contain elements of $\Omega(F)\cup \Omega(F')$.
Similarly, we establish that the first edge from $\Omega(F)\cup \Omega(F')$ 
which occurs in the zigzag $Z$ after the sequence $e'_{2},e'_{3}$ is $e'_{3}$.
The next two edges of $Z$ are $D_{F'}(e'_{3})=e_{1}$ and $D_{F}(e_{1})=e_{2}$, i.e.
the zigzag $Z$ is a cyclic sequence of type
$$\underbrace{e_3, e_1, e'_2}_{F,F'}, X, \underbrace{e'_2, e'_3}_{F'}, Y,  \underbrace{e'_3, e_1, e_2}_{F',F},\dots,$$
where $Y$ is a sequence of edges which does not contains elements of $\Omega(F)\cup \Omega(F')$.
The required statement follows from the second part of Lemma 1\footnote{Note that the next two edges from $\Omega(F)\cup \Omega(F')$ 
contained in the zigzag $Z$ are $M_{F}(e_{2})=e_{2}$ and $D_{F}(e_{2})=e_{3}$.}.
\end{proof}

Suppose that $F_1, F_2,\dots, F_n=F_{1}$, $n\ge 4$ is a simple cycle in the graph ${\mathrm G}_1$. 
Our triangulation is locally $z$-knotted in each $F_{i}$.
Since $F_{i}$ and $F_{i+1}$ are adjacent, we have ${\mathcal Z}(F_{i})={\mathcal Z}(F_{i+1})$.
Therefore,
$${\mathcal Z}(F_{1})=\dots ={\mathcal Z}(F_{n-1})=\{Z,Z^{-1}\}.$$
Applying Lemma \ref{l-2} to the faces $F_{1}$ and $F_2$, we obtain that the face shadow of $Z$ or $Z^{-1}$ is a cyclic sequence of type
$$F_{2},F_{1},\dots,F_{1},\dots,F_{1},F_{2},\dots,F_{2},\dots;$$
in what follows, we will assume that this is the face shadow of $Z$.
The face $F_3$  is adjacent to $F_2$ and we have four possibilities to occurring this face in $Z$ (Fig.6).
\begin{center}
\begin{tikzpicture}[scale=0.6]
\draw (0,0) circle (2cm);
\draw[fill=black] (0:2cm) circle (3pt);
\draw[fill=black] (180:2cm) circle (3pt);

\draw[fill=black] (86:2cm) circle (3pt);
\draw[fill=black] (94:2cm) circle (3pt);
\draw[fill=black] (266:2cm) circle (3pt);
\draw[fill=black] (274:2cm) circle (3pt);

\draw[fill=white] (102:2cm) circle (3pt);
\draw[fill=white] (172:2cm) circle (3pt);
\draw[fill=white] (188:2cm) circle (3pt);
\draw[fill=white] (258:2cm) circle (3pt);

\node at (0:2.5cm) {$F_1$};
\node at (180:2.5cm) {$F_2$};

\node at (83:2.5cm) {$F_1$};
\node at (97:2.5cm) {$F_2$};
\node at (263:2.5cm) {$F_2$};
\node at (277:2.5cm) {$F_1$};

\end{tikzpicture}
\captionof{figure}{ }
\end{center}
Lemma \ref{l-2} shows that 
for the face shadow of $Z$ one of the following possibilities is realized:
$$F_{3},F_{2},F_{1},\dots, F_{1},\dots, F_{1},F_{2},\dots, F_{2},F_{3},\dots,F_{3},\dots$$
or 
$$F_{2},F_{1},\dots,F_{1},\dots,F_{1},F_{2},F_{3},\dots, F_{3},\dots,F_{3},F_{2},\dots$$
(Fig.7a and Fig.7b, respectively).
\begin{center}
\begin{tikzpicture}[scale=0.8]
\draw (0,0) circle (2cm);
\draw[fill=black] (0:2cm) circle (2.8pt);
\draw[fill=black] (180:2cm) circle (2.8pt);

\draw[fill=black] (86:2cm) circle (2.8pt);
\draw[fill=black] (94:2cm) circle (2.8pt);
\draw[fill=black] (266:2cm) circle (2.8pt);
\draw[fill=black] (274:2cm) circle (2.8pt);

\draw[fill=black] (102:2cm) circle (2.8pt);
\draw[fill=black] (172:2cm) circle (2.8pt);
\draw[fill=black] (135:2cm) circle (2.8pt);

\node at (0:2.4cm) {$F_1$};
\node at (180:2.4cm) {$F_2$};

\node at (84:2.35cm) {$F_1$};
\node at (95:2.35cm) {$F_2$};
\node at (264:2.4cm) {$F_2$};
\node at (276:2.4cm) {$F_1$};

\node at (107:2.35cm) {$F_3$};
\node at (171:2.4cm) {$F_3$};
\node at (135:2.4cm) {$F_3$};

\node at (270:3cm) {$(a)$};

\draw[xshift=6cm] (0,0) circle (2cm);
\draw[xshift=6cm,fill=black] (0:2cm) circle (2.8pt);
\draw[xshift=6cm,fill=black] (180:2cm) circle (2.8pt);

\draw[xshift=6cm,fill=black] (86:2cm) circle (2.8pt);
\draw[xshift=6cm,fill=black] (94:2cm) circle (2.8pt);
\draw[xshift=6cm,fill=black] (266:2cm) circle (2.8pt);
\draw[xshift=6cm,fill=black] (274:2cm) circle (2.8pt);

\draw[xshift=6cm,fill=black] (220:2cm) circle (2.8pt);
\draw[xshift=6cm,fill=black] (258:2cm) circle (2.8pt);
\draw[xshift=6cm,fill=black] (172:2cm) circle (2.8pt);

\node[xshift=4.8cm] at (0:2.4cm) {$F_1$};
\node[xshift=4.8cm] at (180:2.4cm) {$F_3$};

\node[xshift=4.8cm] at (84:2.35cm) {$F_1$};
\node[xshift=4.8cm] at (96:2.35cm) {$F_2$};
\node[xshift=4.8cm] at (265:2.4cm) {$F_2$};
\node[xshift=4.8cm] at (276:2.4cm) {$F_1$};

\node[xshift=4.8cm] at (171:2.4cm) {$F_2$};
\node[xshift=4.8cm] at (254:2.4cm) {$F_3$};
\node[xshift=4.8cm] at (220:2.4cm) {$F_3$};

\node[xshift=4.8cm] at (270:3cm) {$(b)$};
\end{tikzpicture}
\captionof{figure}{ }
\end{center}
Apply Lemma \ref{l-2} to the faces $F_{3}$ and $F_{4}$, we establish that $F_4$ always occurs 
in the face shadow of $Z$ after the three times of $F_{1}$.
Recursively, we show that the same holds for every face $F_{i}$ if $3\le i\le n$, i.e.
the face shadow of $Z$  is a sequence of type
$$F_{2},F_{1},\dots,F_{1},\dots, F_{1},F_{2},\dots,F_{i},\dots,F_{i},\dots,F_{i},\dots.$$
In the case when $F_{i}=F_{n}=F_{1}$, this contradicts the fact that the face shadow of $Z$ contains only three times of $F_{1}$.
So, ${\mathrm G}_{1}$ does not contain cycles.

\section{The graph ${\mathrm G}_{2}$ is a forest} 
\begin{lemma}\label{l-3}
Let $F$ and $F'$ be adjacent faces whose $z$-monodromies are of type {\rm (M2)}, i.e. $D_F$ and $D_{F'}$, respectively.
Then there is a unique {\rm(}up to reverse{\rm)} zigzag whose face shadow contains $F$ and $F'$.
This faces shadow is a cyclic sequence of type 
$$F,F',\dots,F,\dots,F',F,\dots,F',\dots$$
(Fig.8)\footnote{As in Lemma \ref{l-2}, the reversed sequence $F',F,\dots,F,'\dots,F,F',\dots,F,\dots$
is the face shadow of the reversed zigzag.}.
\end{lemma}
\begin{center}
\begin{tikzpicture}[scale=0.6]
\draw (0,0) circle (2cm);
\draw[fill=black] (0:2cm) circle (3pt);
\draw[fill=black] (180:2cm) circle (3pt);

\draw[fill=black] (86:2cm) circle (3pt);
\draw[fill=black] (94:2cm) circle (3pt);
\draw[fill=black] (266:2cm) circle (3pt);
\draw[fill=black] (274:2cm) circle (3pt);

\node at (0:2.5cm) {$F$};
\node at (180:2.5cm) {$F'$};

\node at (84:2.5cm) {$F'$};
\node at (96:2.5cm) {$F$};
\node at (264:2.5cm) {$F$};
\node at (276:2.5cm) {$F'$};

\end{tikzpicture}
\captionof{figure}{ }
\end{center}
\begin{proof}
Let $F$ and $F'$ be as in the proof of Lemma \ref{l-2} (see Fig.5).
We have $M_{F}=D_F$ and $M_{F'}=D_{F'}$ which implies that the triangulation $\Gamma$ is locally $z$-knotted in $F$ and $F'$. 
As in the proof of Lemma \ref{l-2}, we obtain that
$${\mathcal Z}(F)={\mathcal Z}(F')=\{Z,Z^{-1}\}$$
and assume that $Z$ is the zigzag containing the sequence $e_{3},e_{1},e'_{2}$.

Let $e$ be the first edge from $\Omega(F)\cup \Omega(F')$ which occurs in the zigzag $Z$ after this sequence.
If $e$ belong to $\Omega(F')$, then $$e=M_{F'}(e'_2)=D_{F'}(e'_2)=e'_3$$ and
the next edge in the zigzag is $D_{F'}(e'_{3})=e_1$.
This means that $M_{F}(e_{1})=e_{1}$ which is impossible, since 
$$M_{F}(e_{1})=D_{F}(e_{1})=e_{2}.$$
Therefore, $e$ belongs to $\Omega(F)$. In this case, we have $e=M_{F}(e_{1})=e_{2}$.
Then the next edge in $Z$ is $D_{F}(e_{2})=e_{3}$ and the zigzag $Z$ is a cyclic sequence of type 
$$e_{3},e_{1},e'_{2},X,e_{2},e_{3},\dots,$$
where $X$ is a sequence of edges which does not contain elements of $\Omega(F)\cup \Omega(F')$.

Now, we determine the first edge $e'$ from  $\Omega(F)\cup \Omega(F')$  which occurs in the zigzag $Z$ after the sequence $e_{2},e_{3}$.
If it belongs to $\Omega(F)$, then 
$$e'=M_{F}(e_{3})=D_{F}(e_{3})=e_{1}$$ 
which is impossible, since we can come to $e_{1}$ by a zigzag only through an element of $\Omega(F)$ or $\Omega(F')$ different from $e_1$.
Hence $e'$ is an element of $\Omega(F')$. Then 
$$e'=M_{F'}(e'_{2})=D_{F'}(e'_{2})=e'_{3}.$$
The next two edges in the zigzag $Z$ are $D_{F'}(e'_{3})=e_{1}$ and $D_{F}(e_{1})=e_{2}$.
So, $Z$ is a cyclic sequence of type
$$\underbrace{e_{3},e_{1},e'_{2}}_{F,F'}, X, \underbrace{e_{2},e_{3}}_{F}, Y, \underbrace{e'_{3},e_{1},e_{2}}_{F',F},\dots,$$
where $Y$ is a sequence of edges which does not contains elements of $\Omega(F)\cup \Omega(F')$.
The second part of Lemma 1 gives the claim\footnote{The next two edges from $\Omega(F)\cup \Omega(F')$ 
contained in the zigzag $Z$ are $M_{F'}(e_{1})=e'_{2}$ and $D_{F'}(e'_{2})=e'_{3}$.}.
\end{proof}

Let $F_1, F_2,\dots, F_n=F_{1}$, $n\ge 4$ be a simple cycle  in the graph ${\mathrm G}_2$. 
Our triangulation is locally $z$-knotted in each $F_{i}$ and, as in the previous section, we have
$${\mathcal Z}(F_{1})=\dots ={\mathcal Z}(F_{n-1})=\{Z,Z^{-1}\}.$$
By Lemma \ref{l-3}, the face shadow of $Z$ or $Z^{-1}$ is a cyclic sequence of type
$$F_1,F_2,\dots,F_1,\dots,F_2,F_1,\dots,F_2,\dots$$
and we suppose that this holds for the face shadow of $Z$.

(1).
First, we consider the case when $F_1, F_2,\dots,F_n$ are consecutive faces in the face shadow of $Z$.
The face $F_3$ is adjacent to $F_2$ and Lemma \ref{l-3} shows that the face shadow of $Z$ is
$$F_1,F_2,F_3,\dots,F_1,\dots,F_2,F_1,\dots,F_3,F_2,\dots,F_3,\dots$$
(Fig.9a).
In the case when $n>4$, 
we apply Lemma \ref{l-3} to the adjacent faces $F_{i-1}$ and $F_i$ with $4\le i\le n-1$. Recursively, we 
establish that $F_{n-1}$ occurs in the face shadow of $Z$ as follows
$$F_1,\dots,F_{n-1},F_1,\dots,F_1,\dots,F_{n-1},\dots,F_{n-1},\dots$$
(Fig.9b).
\begin{center}
\begin{tikzpicture}[scale=0.8]
\draw (0,0) circle (2cm);
\draw[fill=black] (0:2cm) circle (2.8pt);
\draw[fill=black] (180:2cm) circle (2.8pt);

\draw[fill=black] (86:2cm) circle (2.8pt);
\draw[fill=black] (94:2cm) circle (2.8pt);
\draw[fill=black] (266:2cm) circle (2.8pt);
\draw[fill=black] (274:2cm) circle (2.8pt);

\draw[fill=black] (78:2cm) circle (2.8pt);
\draw[fill=black] (172:2cm) circle (2.8pt);
\draw[fill=black] (135:2cm) circle (2.8pt);

\node at (0:2.4cm) {$F_1$};
\node at (180:2.4cm) {$F_3$};

\node at (84:2.35cm) {$F_2$};
\node at (95:2.35cm) {$F_1$};
\node at (264:2.4cm) {$F_1$};
\node at (276:2.4cm) {$F_2$};

\node at (72:2.33cm) {$F_3$};
\node at (170:2.4cm) {$F_2$};
\node at (135:2.45cm) {$F_3$};

\node at (270:3cm) {$(a)$};

\draw[xshift=6cm] (0,0) circle (2cm);
\draw[xshift=6cm,fill=black] (356:2cm) circle (2.8pt);
\draw[xshift=6cm,fill=black] (4:2cm) circle (2.8pt);

\draw[xshift=6cm,fill=black] (90:2cm) circle (2.8pt);

\draw[xshift=6cm,fill=black] (120:2cm) circle (2.8pt);

\draw[xshift=6cm,fill=black] (150:2cm) circle (2.8pt);

\draw[xshift=6cm,fill=black] (270:2cm) circle (2.8pt);

\node[xshift=4.8cm] at (355:2.4cm) {$F_1$};
\node[xshift=4.8cm] at (5:2.65cm) {$F_{n-1}$};

\node[xshift=4.8cm] at (90:2.35cm) {$F_1$};

\node[xshift=4.8cm] at (270:2.4cm) {$F_1$};

\node[xshift=4.8cm] at (120:2.4cm) {$F_{n-1}$};

\node[xshift=4.8cm] at (150:2.6cm) {$F_{n-1}$};

\node[xshift=4.8cm] at (270:3cm) {$(b)$};
\end{tikzpicture}
\captionof{figure}{ }
\end{center}
On the other hand, $F_{n-1}$ is adjacent to $F_{n}=F_{1}$ and, by Lemma \ref{l-3},
the face shadow of $Z$ is a cyclic sequence of type
$$F_1,\dots,F_{n-1},F_1,\dots,F_{n-1},\dots,F_1,F_{n-1},\dots;$$
we get a contradiction.

(2). 
Now, we consider the case when $F_1, F_2,\dots,F_n$ are not consecutive faces in the face shadow of $Z$.
Let us take the greatest number $k$ such that $F_1,\dots, F_k$ are consecutive faces in the face shadow of $Z$.
Since this face shadow contains the sequence $F_{1},F_{2}$, we have $k\ge 2$.
In the case when $k=2$, Lemma \ref{l-3} shows that the face shadow of $Z$ is a cyclic sequence of type
$$F_{1},F_{2},\dots, F_{1},\dots,F_{3},F_{2},F_{1},\dots,F_{3},\dots,F_{2},F_{3},\dots$$
(Fig.10).  This means that $F_{1},F_{2},F_{3}$ are consecutive faces in the face shadow of the reversed zigzag $Z^{-1}$.
So, we can assume that $k\ge 3$.

\begin{center}
\begin{tikzpicture}[scale=0.8]
\draw (0,0) circle (2cm);
\draw[fill=black] (0:2cm) circle (2.8pt);
\draw[fill=black] (180:2cm) circle (2.8pt);

\draw[fill=black] (86:2cm) circle (2.8pt);
\draw[fill=black] (94:2cm) circle (2.8pt);
\draw[fill=black] (266:2cm) circle (2.8pt);
\draw[fill=black] (274:2cm) circle (2.8pt);

\draw[fill=black] (282:2cm) circle (2.8pt);
\draw[fill=black] (172:2cm) circle (2.8pt);
\draw[fill=black] (225:2cm) circle (2.8pt);

\node at (0:2.4cm) {$F_1$};
\node at (180:2.4cm) {$F_2$};

\node at (84:2.35cm) {$F_2$};
\node at (95:2.35cm) {$F_1$};
\node at (264:2.4cm) {$F_1$};
\node at (276:2.4cm) {$F_2$};

\node at (287:2.38cm) {$F_3$};
\node at (170:2.4cm) {$F_3$};
\node at (225:2.4cm) {$F_3$};

\end{tikzpicture}
\captionof{figure}{ }
\end{center}
We apply Lemma \ref{l-3} to the faces $F_{k-1}$ and $F_{k}$ and to the faces $F_{k}$ and $F_{k+1}$.
Taking in account the fact that $F_{k+1}$ does not occur in the face shadow of $Z$ immediately after $F_{1},\dots,F_{k}$,
we obtain that the face shadow of $Z$ is a cyclic sequence of type
$$F_1,F_2,\dots, F_{k-1},F_k,\dots, F_1,\dots, F_2,F_1,\dots,$$
$$F_{k+1},F_k,F_{k-1},\dots,F_{k+1},\dots,F_k,F_{k+1},\dots$$
(Fig.11).
\begin{center}
\begin{tikzpicture}[scale=1]
\draw (0,0) circle (2cm);
\draw[fill=black] (0:2cm) circle (2.6pt);
\draw[fill=black] (180:2cm) circle (2.6pt);

\draw[fill=black] (86:2cm) circle (2.6pt);
\draw[fill=black] (94:2cm) circle (2.6pt);
\draw[fill=black] (266:2cm) circle (2.6pt);
\draw[fill=black] (274:2cm) circle (2.6pt);

\draw[fill=black] (41:2cm) circle (2.6pt);
\draw[fill=black] (49:2cm) circle (2.6pt);

\draw[fill=black] (150:2cm) circle (2.6pt);

\draw[fill=black] (188:2cm) circle (2.6pt);
\draw[fill=black] (196:2cm) circle (2.6pt);

\draw[fill=black] (124:2cm) circle (2.6pt);
\draw[fill=black] (116:2cm) circle (2.6pt);

\node at (0:2.35cm) {$F_1$};
\node at (180:2.5cm) {$F_{k-1}$};

\node at (85.5:2.3cm) {$F_2$};
\node at (94.5:2.3cm) {$F_1$};
\node at (265.5:2.35cm) {$F_1$};
\node at (274.5:2.35cm) {$F_2$};

\node at (38:2.35cm) {$F_k$};
\node at (47:2.4cm) {$F_{k-1}$};

\node at (147:2.4cm) {$F_{k+1}$};

\node at (189:2.33cm) {$F_k$};
\node at (196:2.45cm) {$F_{k+1}$};

\node at (126.5:2.35cm) {$F_k$};
\node at (113.5:2.3cm) {$F_{k+1}$};

\end{tikzpicture}
\captionof{figure}{ }
\end{center}
Finally, for every $i$ satisfying $k<i\le n$ we establish that $F_{i}$ occurs in the face shadow of $Z$ as follows
$$F_{1},F_{2},\dots,F_{1},\dots, F_{2},F_{1},\dots,F_{i},\dots,F_{i},\dots,F_{i},\dots;$$
if we take $F_i=F_n=F_1$,  then the latter shows that $F_{1}$ is contained in the face shadow of $Z$ 
more than $3$ times which is impossible.

In each of the considered above cases, we get a contradictions which means that ${\mathrm G}_2$ does not contain cycles. 

\section{Two examples}
Let $BP_n$ be the $n$-gonal bipyramid 
containing an $n$-gone whose vertices  are denoted by $1,\dots,n$ and connected with two disjoint vertices $a, b$.
We also consider the $n$-gonal  bipyramid $BP'_n$, where the vertices of the $n$-gone are denoted by $1',\dots,n'$ and 
$a',b'$ are the remaining two vertices.

\begin{exmp}{\rm
The $3$-gonal bipyramids $PB_{3}$ and $PB'_{3}$ are $z$-knotted triangulations whose zigzags are the cyclic sequences
$$\underbrace{12, 2b, b3, 31, 1a}_{A},\underbrace{a2, 23, 3b, b1, 12}_{B}, \underbrace{2a, a3, 31, 1b, b2, 23, 3a, a1}_{C}$$
and
$$\underbrace{1'2', 2'b', b'3', 3'1', 1'a'}_{A'},\underbrace{a'2', 2'3', 3'b', b'1', 1'2'}_{B'}, \underbrace{2'a', a'3', 3'1', 1'b', b'2', 2'3', 3'a', a'1'}_{C'}.$$
Let $S$ and $S'$ be faces of $BP_3$ and $BP'_3$ containing $a,1,2$ and $a',1',2'$  (respectively).
Consider the connected sum $BP_{3} \#_{g} BP'_{3}$, where $g:\partial S\to \partial S'$ satisfies 
$$g(a)=a',\;\;g(1)=1',\;\;g(2)=2'$$
(Fig.12).
This connected sum is $z$-knotted and the unique zigzag (up to reverse) is the cyclic sequence
$$Z=\{A, C'^{-1}, B, A', C^{-1}, B'\},$$
where $C^{-1}$ and $C'^{-1}$ are the sequences reversed to $C$ and $C'$ (respectively);
note that for any two consecutive parts $X,Y$ in $Z$ the last edge from $X$ is identified with the first edge from $Y$.
Denote by $F$ and $F'$ the faces of $BP_{3} \#_{g} BP'_{3}$ containing $b,1,2$ and $b',1,2$.
By \cite[Example 6]{PT2}, the $z$-monodromies of these faces are of type (M2).
Each of the remaining $8$ faces contains one of the edges $23,31,23',3'1$.
The edge $23$ is contained in $B$ and $C$, i.e. the zigzag $Z$ passes through this edge twice in different directions. 
The same holds for the three other edges. 
By the remark at the end of Section 3, for each face containing one of these four edges the $z$-monodromy is of type (M3) or (M4).
Therefore, the graph ${\mathrm G}_2$ corresponding to $BP_{3} \#_{g} BP'_{3}$ is a linear graph $P_2$.
The same is true for the similarly defined connected sum of $(2k+1)$-gonal and $(2k'+1)$-gonal bipyramids 
for arbitrary odd $k$ and $k'$ (using \cite[Example 6]{PT2} the readers can check all the details).
}\end{exmp}

\begin{center}
\begin{tikzpicture}[scale=0.8]

\draw[fill=black] (0,0) circle (2pt);
\draw[fill=black] (2,0) circle (2pt);
\draw[fill=black] (1,1) circle (2pt);
\draw[fill=black] (3,1) circle (2pt);
\draw[fill=black] (1.5,-1.5) circle (2pt);
\draw[fill=black] (0.75,2.5) circle (2pt);
\draw[fill=black] (2.25,2.5) circle (2pt);

\draw[thick] (0,0) -- (2,0)-- (3,1)-- (1,1)-- (0,0)-- (0.75,2.5);
\draw[thick] (2,0) -- (1,1) -- (0.75,2.5) -- (2,0) -- (2.25,2.5)-- (3,1);
\draw[thick] (2.25,2.5)-- (1,1);
\draw[thick] (0,0)-- (1.5,-1.5) -- (2,0);
\draw[thick] (1,1)-- (1.5,-1.5) -- (3,1);

\end{tikzpicture}
\captionof{figure}{ }
\end{center}

\begin{exmp}{\rm
Each of the $6$-gonal bipyramids $BP_6$ and $BP'_6$ contains exactly two zigzags (up to reverse).
The cyclic sequences
$$\underbrace{12, 2b, b3, 34, 4a, a5, 56, 6b, b1, 12}_{A},\underbrace{2a, a3, 34, 4b, b5, 56, 6a, a1}_{B}$$
and
$$\underbrace{a2, 23, 3b, b4, 45, 5a, a6, 61, 1b, b2, 23, 3a, a4, 45, 5b, b6, 61, 1a}_{C}$$
are zigzags in $BP_6$. Similarly, 
$$\underbrace{1'2', 2'b', b'3', 3'4', 4'a', a'5', 5'6', 6'b', b'1', 1'2'}_{A'},\underbrace{2'a', a'3', 3'4', 4'b', b'5', 5'6', 6'a', a'1'}_{B'}$$
and
$$\underbrace{a'2', 2'3', 3'b', b'4', 4'5', 5'a', a'6', 6'1', 1'b', b'2', 2'3', 3'a', a'4', 4'5', 5'b', b'6', 6'1', 1'a'}_{C'}$$
are zigzags in $BP'_6$.
Let $S$ and $S'$ be faces of $BP_6$ and $BP'_6$  containing vertices $a,1,2$ and $a',1',2'$ (respectively).
Let also $g:\partial S\to \partial S'$ be the special homeomorphism satisfying 
$$g(a)=2',\; g(1)=a',\; g(2)=1'.$$
By \cite[Example 7]{PT2}, the connected sum $BP_{6} \#_{g} BP'_{6}$ is $z$-knotted and the unique zigzag (up to reverse)
is the cyclic sequence
$$Z=\{A, C'^{-1}, C^{-1}, A', B, B'\}$$
(as in the previous example, for any two consecutive parts $X,Y$ in $Z$ the last edge from $X$ is identified with the first edge from $Y$).
We will need the following observations concerning the edges of $BP_{6}$:
\begin{enumerate}
\item[$\bullet$] 
Each of the edges $12, 34, 56$ is contained twice in the sequence $A, B$ and each of the edges $23, 45, 61$ is contained twice in the sequence $C$.
\item[$\bullet$] 
An edge $e$ containing $a$ or $b$ belongs to the sequence $A,B$ if and only if $-e$ is contained in $C$.
\end{enumerate}
The same holds for the edges of $BP'_6$. 
Therefore, the zigzag $Z$ passes through each edge of $BP_{6} \#_{g} BP'_{6}$ twice in the same direction. 
Then for each face of the connected sum $BP_{6} \#_{g} BP'_{6}$ the $z$-monodromy is of type (M1) or (M2),
see the remark at the end of Section 3.
A direct verification shows that the $z$-monodromies of the faces
$$a23, a34, a61, a13', b45, b56,  15'6', 126', b'3'4', b'4'5'$$
and the faces
$$a45, a56, 13'4', 14'5', 
b61, b12, b23, b34,  
ab'3', ab'2, b'26', b'5'6'$$
are of types (M1) and (M2), respectively.
So, the graph ${\rm G}_1$ is a linear forest formed by five $P_{2}$;
the graph ${\rm G}_2$ is a linear forest consisting of two $P_2$ and two $P_4$.
}\end{exmp}

\end{document}